\documentclass[12pt,twoside]{amsart}
\usepackage{amssymb,amscd,amsxtra,calc,mathrsfs}
\usepackage{amsmath}
\usepackage{xypic}

\title{On log K-stability for asymptotically log Fano varieties}
\author{Kento Fujita} 
\date{\today}
\address{Department of Mathematics, Faculty of Science, 
Kyoto University, Kyoto 606-8502, Japan}
\email{fujita@math.kyoto-u.ac.jp}

\newcommand{\pr}{\mathbb{P}}

\newcommand{\Z}{\mathbb{Z}}
\newcommand{\Q}{\mathbb{Q}}
\newcommand{\R}{\mathbb{R}}
\newcommand{\C}{\mathbb{C}}

\newcommand{\A}{\mathbb{A}}
\newcommand{\G}{\mathbb{G}}

\newcommand{\Image}{\operatorname{Image}}

\newcommand{\DIV}{\operatorname{div}}

\newcommand{\id}{\operatorname{id}}

\newcommand{\DF}{\operatorname{DF}}

\newcommand{\vol}{\operatorname{vol}}
\newcommand{\sI}{\mathcal{I}}

\newcommand{\sO}{\mathcal{O}}

\newcommand{\sX}{\mathcal{X}}
\newcommand{\sD}{\mathcal{D}}
\newcommand{\sL}{\mathcal{L}}

\newcommand{\sM}{\mathcal{M}}

\newtheorem{thm}{Theorem}[section]
\newtheorem{lemma}[thm]{Lemma}

\newtheorem{corollary}[thm]{Corollary}

\theoremstyle{definition}
\newtheorem{definition}[thm]{Definition}
\newtheorem{remark}[thm]{Remark}

\newtheorem{assumption}[thm]{Assumption}
\newtheorem{example}[thm]{Example}
\newtheorem*{ack}{Acknowledgments}

\begin{document}

\maketitle 

\begin{abstract}
The notion of asymptotically log Fano varieties was given by Cheltsov and Rubinstein. 
We show that, 
if an asymptotically log Fano variety $(X, D)$ satisfies that $D$ is irreducible
and $-K_X-D$ is big, then $X$ does not admit K\"ahler-Einstein edge metrics 
with angle $2\pi\beta$ along $D$ 
for any sufficiently small positive rational number $\beta$. 
This gives an affirmative answer to a conjecture of Cheltsov and Rubinstein. 
\end{abstract}

\setcounter{tocdepth}{1}
\tableofcontents

\section{Introduction}\label{intro_section}

The purpose of this article is to give a simple necessary criterion for log K-stability of 
$((X, D), -K_X-(1-\beta)D)$ with cone angle $2\pi\beta$ in the sense of \cite{OS}, 
where $X$ is projective log terminal and 
$D$ is a reduced Weil divisor with $-K_X-(1-\beta)D$ ample. 
The motivation comes from a recent preprint of Cheltsov and Rubinstein \cite{CR}, 
who treated the case that the dimension of $X$ is equal to two. 
In this article, we show the following result. 

\begin{thm}[{=Theorem \ref{thmthm}}]\label{mainthm}
Let $X$ be a normal projective variety which is log terminal, 
$D$ be a nonzero reduced Weil divisor on $X$ which is $\Q$-Cartier, 
and $0\leq\beta\leq 1$ be a rational number. 
Assume that the pair $(X, (1-\beta)D)$ is dlt, $-K_X-(1-\beta)D$ is ample, 
and $((X, D), -K_X-(1-\beta)D)$ is log K-stable $($resp.\ log K-semistable$)$ with 
cone angle $2\pi\beta$. Then we have $\eta_\beta(D)>0$ $($resp.\ $\geq 0)$, 
where 
\[
\eta_\beta(D):=\beta\cdot\vol_X(-K_X-(1-\beta)D)-\int_0^\infty
\vol_X(-K_X-(1-\beta+x)D)dx.
\]
Note that $\vol_X$ is the volume function $($see \cite{L}$)$. 
\end{thm}

Theorem \ref{mainthm} immediately gives the following corollary. 

\begin{corollary}[{see Corollary \ref{corcor}}]\label{intro_cor}
Let $X$ be a smooth projective variety and $D$ be a nonzero 
reduced simple normal crossing divisor on $X$. 
Assume that $-K_X-(1-\beta)D$ is ample for any $0<\beta\ll 1$ and 
the divisor $-K_X-D$ is big. Then $((X, D), -K_X-(1-\beta)D)$ is not log K-semistable 
with cone angle $2\pi\beta$ for any $0<\beta\ll 1$ with $\beta\in\Q$. In particular, 
$X$ does not admit K\"ahler-Einstein edge metrics with angle $2\pi\beta$ along $D$ 
for any $0<\beta\ll 1$ with $\beta\in\Q$. 
\end{corollary}

Corollary \ref{intro_cor} gives an affirmative answer for a conjecture of Cheltsov and 
Rubinstein for asymptotically log Fano varieties \cite{CR0} with irreducible boundaries 
in any dimension. Although the following corollary is a special case of 
Corollary \ref{intro_cor}, we state the assertion for the readers' convenience. 

\begin{corollary}[{see \cite[Conjecture 1.11 (i)]{CR0}}]\label{maincor}
Let $(X, D)$ be an asymptotically log Fano variety with $D$ irreducible, 
that is, $X$ is a smooth projective variety and $D$ is a smooth irreducible divisor on 
$X$ such that $-K_X-(1-\beta)D$ is ample for any $0<\beta\ll 1$. 
If the divisor $-K_X-D$ is big, 
then $X$ does not admit K\"ahler-Einstein edge metrics with angle 
$2\pi\beta$ along $D$ for any $0<\beta\ll 1$ with $\beta\in\Q$. 
\end{corollary}

\begin{remark}\label{r_rmk}
In \cite[Theorem 1.6]{CR} (see also \cite[Conjecture 1.5]{CR}), 
Cheltsov and Rubinstein proved Corollary \ref{maincor} 
in dimension two by using a construction of flops on the deformation to the normal cone. 
\end{remark}

The strategy for the proof of Theorem \ref{mainthm} 
is essentially same as the strategy in \cite{fjt}. 
We consider a kind of ``log-version" of divisorial stability along $D$ in the sense of 
\cite{fjt}. 
We construct a specific log semi test configuration from certain section ring (see 
Remark \ref{BCHM_rmk}) and 
calculate its log Donaldson-Futaki invariant explicitly by using the theory of 
``geography of models" (see Theorem \ref{KKL_thm}). 

\begin{ack}
The author thanks Professor Yanir Rubinstein who recommended to write down 
Section \ref{ex_section}. 
The author acknowledges support from the Simons Center for Geometry and Physics, 
Stony Brook University at which some of the research for this paper was performed.
The author is partially supported by a JSPS Fellowship for Young Scientists. 
\end{ack}

A \emph{variety} stands for a reduced, irreducible, separated and of 
finite type scheme over the complex number field $\C$. For the theory 
of minimal model program, we refer the readers to \cite{KoMo}. 
For any Weil divisor $E$ on a normal variety $X$, the 
\emph{divisorial sheaf} on $X$ is denoted by $\sO_X(E)$. More precisely, 
the section $\Gamma(U, \sO_X(E))$ on any open subscheme $U\subset X$ is given by 
the following:
\[
\{f\in \C(X)\,|\,\DIV(f)|_U+E|_U\geq 0\},
\]
where $\C(X)$ is the function field of $X$. 

For varieties $X_1$ and $X_2$, let $p_i\colon X_1\times X_2\to X_i$ $(i=1$, $2)$ be 
the projection morphisms.

\section{Log K-stability}\label{K_section}

We recall the definition of log K-stability. 

\begin{definition}[{see \cite{OS}}]\label{K_dfn}
Let $X$ be an $n$-dimensional normal projective variety, $L$ be an ample line bundle 
on $X$, and $D$ be a reduced Weil divisor on $X$. 
\begin{enumerate}
\renewcommand{\theenumi}{\arabic{enumi}}
\renewcommand{\labelenumi}{(\theenumi)}
\item\label{K_dfn1}
A coherent ideal sheaf $\sI\subset\sO_{X\times\A_t^1}$ is said to be 
a \emph{flag ideal}
if $\sI$ is of the form 
\[
\sI=I_M+I_{M-1}t^1+\cdots+I_1t^{M-1}+(t^M)\subset\sO_{X\times\A_t^1},
\]
where $I_M\subset\cdots\subset I_1\subset\sO_X$ is a sequence of 
coherent ideal sheaves of $X$. 
\item\label{K_dfn2}
Let $m\in\Z_{>0}$, and let $\sI\subset\sO_{X\times\A^1}$ be a flag ideal. 
A \emph{log semi test configuration} $((\sX, \sD), \sM)/\A^1$ \emph{of} 
$((X, D), L^{\otimes m})$ \emph{obtained by} $\sI$ is given from the following data: 
\begin{itemize}
\item
$\Pi\colon\sX\to X\times\A^1$ is the blowing up along $\sI$, 
$\sD\subset\sX$ is given by the blowing up of $D\times\A^1$ along 
$\sI|_{D\times\A^1}$,
and $E\subset\sX$ 
is the Cartier divisor defined by $\sO_\sX(-E)=\sI\cdot\sO_\sX$,
\item
$\sM$ is the line bundle on $\sX$ defined by 
$\sM:=\Pi^*p_1^*L^{\otimes m}\otimes\sO_\sX(-E)$, 
\end{itemize}
such that we require the following: 
\begin{itemize}
\item
$\sI$ is not of the form $(t^M)$, and
\item
$\sM$ is semiample  over $\A^1$.
\end{itemize}
\item\label{K_dfn3}
Assume that $((\sX, \sD), \sM)/\A^1$ is a log semi test configuration of 
$((X, D), L^{\otimes m})$
obtained by $\sI$. Then the multiplicative group 
$\G_m$ naturally acts on $(\sX, \sM)$ and $(\sD, \sM|_\sD)$. 
For $k\in\Z_{>0}$, let 
$w(k)$ be the total weight of $\G_m$-action on $H^0(\sX_0, \sM^{\otimes k}|_{\sX_0})$ 
and $\tilde{w}(k)$ be the total weight of $\G_m$-action on 
$H^0(\sD_0, \sM^{\otimes k}|_{\sD_0})$, where $\sX_0\subset\sX$ and 
$\sD_0\subset\sD$ are the scheme-theoretic fibers at $0\in\A^1$, respectively. 
It is known that, for $k\gg 0$,  $w(k)$ (resp.\ $\tilde{w}(k)$) is a polynomial function 
of degree at most $n+1$ (resp.\ $n$). For $k\gg 0$, we set 
\begin{eqnarray*}
\chi(X, L^{\otimes mk})&=&a_0k^n+a_1k^{n-1}+O(k^{n-2}),\\
\chi(D, L|_D^{\otimes mk})&=&\tilde{a}_0k^{n-1}+O(k^{n-2}),\\
w(k)&=&b_0k^{n+1}+b_1k^n+O(k^{n-1}),\\
\tilde{w}(k)&=&\tilde{b}_0k^{n}++O(k^{n-1}).
\end{eqnarray*}
For any $\beta\in[0,1]$, we set the 
\emph{log Donaldson-Futaki invariant} $\DF_\beta((\sX, \sD), \sM)$ 
\emph{with cone angle $2\pi\beta$} as
\begin{eqnarray*}
\DF_\beta((\sX, \sD), \sM):=2(b_0a_1-b_1a_0)+(1-\beta)(a_0\tilde{b}_0-b_0\tilde{a}_0).
\end{eqnarray*}
\item\label{K_dfn4}
Let $\beta\in[0,1]$. 
$((X, D),L)$ is said to be \emph{log K-stable} (resp.\ \emph{log K-semistable}) 
\emph{with cone angle $2\pi\beta$}
if 
$\DF_\beta((\sX, \sD), \sM)>0$ (resp.\ $\geq 0$) holds for any $m\in\Z_{>0}$, 
for any flag ideal $\sI$, and for any log semi test configuration 
$((\sX, \sD), \sM)/\A^1$ of $((X, D), L^{\otimes m})$ obtained by $\sI$. 
For an ample $\Q$-divisor $A$ on $X$, 
$((X, D), A)$ is said to be \emph{log K-stable} (resp.\ \emph{log K-semistable}) 
\emph{with cone angle $2\pi\beta$} if $((X, D), \sO_X(aA))$ is so for some $a\in\Z_{>0}$
with $aA$ Cartier (this definition does not depend on the choice of $a$). 
\end{enumerate}
\end{definition}

The following theorem is important. 

\begin{thm}[{see \cite{Berman, CR, OS}}]\label{agdg_thm}
Let $X$ be a smooth projective variety, $D$ be a reduced simple normal crossing 
divisor on $X$, and let $\beta\in[0, 1]\cap\Q$. 
Assume that $-K_X-(1-\beta)D$ is ample 
and $X$ admits K\"ahler-Einstein edge metrics with angle $2\pi\beta$ along $D$. 
Then $((X, D), -K_X-(1-\beta)D)$ is log K-semistable with cone angle $2\pi\beta$. 
\end{thm}

\section{Constructing log semi test configurations}\label{tc_section}

In this section, from a pair $(X, D)$,
we construct a specific log semi test configuration via $D$. 
The construction is essentially in the same way as in \cite[\S 3]{fjt}. 
We fix the following condition: 

\begin{assumption}\label{fg_assump}
Let $X$ be an $n$-dimensional normal projective variety which is log terminal, $D$ 
is a nonzero reduced Weil divisor on $X$ 
which is $\Q$-Cartier, and $\beta\in[0, 1]\cap\Q$. 
Assume that the pair $(X, (1-\beta)D)$ is dlt and $-K_X-(1-\beta)D$ is ample.
\end{assumption}

\begin{definition}\label{tau_dfn}
Under Assumption \ref{fg_assump}, we set
\begin{eqnarray*}
\tau(D)&:=&\sup\{\tau\in\R_{>0}\,|\,-K_X-\tau D\text{ big}\},\\
\tau_\beta(D)&:=&\sup\{\tau\in\R_{>0}\,|\,-K_X-(1-\beta+\tau)D\text{ big}\}.
\end{eqnarray*}
It is obvious that $\tau_\beta(D)=\tau(D)-(1-\beta)$. 
We remark that $\tau(D)>1$ holds if and only if the divisor $-K_X-D$ is big. 
\end{definition}

\begin{remark}\label{BCHM_rmk}
By \cite[Corollary 1.1.9]{BCHM}, the $\C$-algebra
\[
\bigoplus_{\substack{k\in\Z_{\geq 0}\\ j\in\Z_{\geq 0}}}
H^0(X, \sO_X(\lfloor k(-K_X-(1-\beta)D)-jD\rfloor))
\]
is finitely generated, where $\lfloor k(-K_X-(1-\beta)D)-jD\rfloor$ is the biggest 
$\Z$-divisor which is contained by $k(-K_X-(1-\beta)D)-jD$. 
We note that 
$H^0(X, \sO_X(\lfloor k(-K_X-(1-\beta)D)-jD\rfloor))=0$ if $j>k\tau_\beta(D)$. 
Thus, there exists $r\in\Z_{>0}$ such that 
\begin{itemize}
\item
$L_\beta:=r(-K_X-(1-\beta)D)$ is Cartier, and 
\item
the $\C$-algebra
\[
\bigoplus_{\substack{k\in\Z_{\geq 0}\\ j\in[0,kr\tau_\beta(D)]\cap\Z}}
H^0(X, \sO_X(kL_\beta-jD))
\]
is generated by 
\[
\bigoplus_{j\in[0,r\tau_\beta(D)]\cap\Z}
H^0(X, \sO_X(L_\beta-jD)).
\]
\end{itemize}
From now on, we fix such $r$ (and $L_\beta$). 
\end{remark}

\begin{thm}[{\cite[Theorem 4.2]{KKL}}]\label{KKL_thm}
Under Assumption \ref{fg_assump}, 
there exist 
\begin{itemize}
\item
a sequence of rational numbers
\[
0=\tau_0<\tau_1<\cdots<\tau_m=\tau_\beta(D),
\]  
\item
normal projective varieties
$X_1,\dots,X_m$ such that $X_1=X$, and 
\item
mutually distinct birational contraction maps
$\phi_i\colon X\dashrightarrow X_i$  
with $\phi_1=\id_X$ $(1\leq i\leq m)$ 
\end{itemize}
such that the following hold: 
\begin{itemize}
\item
for any $x\in[\tau_{i-1}, \tau_i]$, $\phi_i$ is 
a semiample model 
$($see \cite[Definition 2.3]{KKL}$)$ of $-K_X-(1-\beta+x)D$, and 
\item
if $x\in(\tau_{i-1}, \tau_i)$, then $\phi_i$ is 
the ample model 
$($see \cite[Definition 2.3]{KKL}$)$ of $-K_X-(1-\beta+x)D$. 
\end{itemize}
\end{thm}

\begin{proof}
By \cite[Corollary 1.4.3]{BCHM}, there exists a projective birational morphism 
$\sigma\colon\tilde{X}\to X$ such that $\sigma$ is an isomorphism in codimension 
one and $\tilde{X}$ is $\Q$-factorial. Let $\tilde{D}$ be the strict transform of $D$ 
on $\tilde{X}$. 
A semiample model (resp.\ the ample model) of 
$-K_{\tilde{X}}-(1-\beta+x)\tilde{D}$ is a semiample model (resp.\ the ample model) 
of $-K_X-(1-\beta+x)D$. 
Moreover, the $\C$-algebra
\[
\bigoplus_{\substack{k\in\Z_{\geq 0}\\ j\in\Z_{\geq 0}}}
H^0(X, \sO_X(\lfloor k(-K_X-(1-\beta)D)-jD\rfloor))
\]
is equal to the $\C$-algebra
\[
\bigoplus_{\substack{k\in\Z_{\geq 0}\\ j\in\Z_{\geq 0}}}
H^0(\tilde{X}, \sO_{\tilde{X}}(\lfloor k(-K_{\tilde{X}}-(1-\beta)\tilde{D})-j\tilde{D}\rfloor)).
\]
Thus we can apply \cite[Theorem 4.2]{KKL}. 
\end{proof}

We construct a log semi test configuration of $((X, D), \sO_X(L_\beta))$ from $D$.
For any $j\in[0,r\tau_\beta(D)]\cap\Z$, we set 
\begin{eqnarray*}
I_j:=\Image(H^0\left(X, \sO_X(L_\beta-jD)\right)\otimes_\C \sO_X(-L_\beta)\to\sO_X),
\end{eqnarray*}
where the homomorphism is the evaluation. 
Note that, for any $j\in[0,r\tau_\beta(D)]\cap\Z$, $I_j\subset\sO_X(-jD)$ and
\[
0\subset I_{r\tau_\beta(D)}\subset\cdots\subset I_1\subset I_0=\sO_X
\]
hold. 
For $k\in\Z_{>0}$ and $j\in[0, kr\tau_\beta(D)]\cap \Z$, we define 
\[
J_{(k, j)}:=\sum_{\substack{j_1+\cdots+j_k=j,\\ 
j_1,\dots,j_k\in[0,r\tau_\beta(D)]\cap\Z}}I_{j_1}\cdots I_{j_k}
\subset\sO_X.
\]

\begin{lemma}[{see \cite[Lemma 3.3]{fjt}}]\label{gen_lem}
The $J_{(k, j)}\subset\sO_X$ is equal to 
\begin{eqnarray*}
\Image(H^0\left(X, \sO_X(kL_\beta-jD)\right)\otimes_\C 
\sO_X(-kL_\beta)\to\sO_X).
\end{eqnarray*}
In particular, we have
\[
H^0\left(X, \sO_X(kL_\beta-jD)\right)
=H^0\left(X, \sO_X(kL_\beta)\cdot J_{(k, j)}\right).
\]
\end{lemma}

\begin{proof}
Set 
\[
V_{k,j}:=H^0(X, \sO_X(kL_\beta-jD))
\]
for simplicity. We remark that, by the choice of $r\in\Z_{>0}$, the homomorphism 
\[
\bigoplus_{\substack{j_1+\cdots+j_k=j,\\ j_1,\dots,j_k\in[0,r\tau_\beta(D)]\cap\Z}}
V_{1, j_1}\otimes_\C\cdots\otimes_\C V_{1, j_k}\to V_{k, j}
\]
is surjective. For any $1\leq i\leq k$, 
the ideal sheaf $I_{j_i}$ is nothing but 
\[
\Image(V_{1, j_i}\otimes_\C\sO_X(-L_\beta)\to \sO_X).
\]
Thus the assertion follows. 
\end{proof}

Set the flag ideal $\sI$ such that
\[
\sI:=I_{r\tau_\beta(D)}+I_{r\tau_\beta(D)-1}t^1+\cdots+ I_1t^{r\tau_\beta(D)-1}+(t^{r\tau_\beta(D)})
\subset\sO_{X\times\A^1_t}.
\]
For any $k\in\Z_{>0}$, we have
\[
\sI^k=J_{(k, kr\tau_\beta(D))}+J_{(k, kr\tau_\beta(D)-1)}t^1+
\cdots+J_{(k, 1)}t^{kr\tau_\beta(D)-1}+(t^{kr\tau_\beta(D)})
\]
by the construction of $J_{(k,j)}$.
Let $\Pi\colon\sX\to X\times\A^1$ be the blowing up along $\sI$ and 
let $E\subset\sX$ be the Cartier divisor given by the equation 
$\sO_\sX(-E)=\sI\cdot\sO_\sX$. Set 
$\sL_\beta:=\Pi^*p_1^*\sO_X(L_\beta)\otimes\sO_\sX(-E)$. 
Let $\sD\to D\times\A^1$ be the blowing up along $\sI|_{D\times\A^1}$. 
We note that $\sI|_{D\times\A^1}=(t^{r\tau_\beta(D)})\subset\sO_{D\times\A^1}$ since 
$I_j\subset\sO_X(-jD)\subset\sO_X(-D)$ for any $j>0$. 
In particular, $\sD\simeq D\times\A^1$ holds.

\begin{lemma}[{see \cite[Lemma 3.4]{fjt}}]\label{stc_lem}
$((\sX, \sD), \sL_\beta)/\A^1$ is a log semi test configuration of 
$((X, D), L_\beta)$. 
\end{lemma}

\begin{proof}
Set $\alpha:=p_2\circ\Pi\colon\sX\to\A^1$. 
It is enough to check that $\sL_\beta$ is $\alpha$-semiample. 
By Lemma \ref{gen_lem}, the homomorphism 
\[
H^0(X, \sO_X(kL_\beta)\cdot J_{(k, j)})\otimes_\C \sO_X\to
\sO_X(kL_\beta)\cdot J_{(k, j)}
\]
is surjective for any $k\in\Z_{>0}$ and $j\in[0, kr\tau_\beta(D)]\cap\Z$. Thus 
\[
H^0\left(X\times\A^1, p_1^*\sO_X(kL_\beta)\cdot\sI^k\right)\otimes_{\C[t]}
\sO_{X\times\A^1}\to p_1^*\sO_X(kL_\beta)\cdot \sI^k
\]
is surjective for any $k\in\Z_{>0}$. 
From \cite[Lemma 5.4.24]{L}, we have 
\begin{eqnarray*}
\alpha^*\alpha_*\sL_\beta^{\otimes k}&\simeq& 
\Pi^*(p_2)^*(p_2)_*(p_1^*\sO_X(kL_\beta)\cdot\sI^k)\\
&=&\Pi^*\left(H^0\left(X\times\A^1, p_1^*\sO_X(kL_\beta)\cdot\sI^k\right)
\otimes_{\C[t]}\sO_{X\times\A^1}\right)\\
&\twoheadrightarrow&\Pi^*\left(p_1^*\sO_X(kL_\beta)\cdot\sI^k\right)\\
&\twoheadrightarrow&
\Pi^*p_1^*\sO_X(kL_\beta)\otimes\sO_\sX(-kE)=\sL_\beta^{\otimes k} 
\end{eqnarray*}
for $k\gg 0$. 
\end{proof}

\begin{definition}\label{divst_dfn}
We say the log semi test configuration $((\sX, \sD), \sL_\beta)/\A^1$ 
the 
\emph{basic log semi test configuration of} $((X, D), \sO_X(L_\beta))$ \emph{via} $D$. 
\end{definition}

Now we calculate the log Donaldson-Futaki invariant of the basic log semi 
test configuration $((\sX, \sD), \sL_\beta)/\A^1$ of 
$((X, D), \sO_X(L_\beta))$ via $D$.
By the asymptotic Riemann-Roch theorem, we have
\begin{eqnarray*}
a_0=\frac{(L_\beta^{\cdot n})}{n!},\quad a_1=\frac{(L_\beta^{\cdot n-1}\cdot -K_X)}
{2\cdot(n-1)!}, \quad \tilde{a}_0=\frac{(L_\beta^{\cdot n-1}\cdot D)}{(n-1)!}.
\end{eqnarray*}
(We follow the notation in Definition \ref{K_dfn}.)
By \cite[\S 3]{odk}, 
\begin{eqnarray*}
w(k)&=&-\dim\left(\frac{H^0(X\times\A^1, p_1^*\sO_X(kL_\beta))}{H^0
(X\times\A^1, p_1^*\sO_X(kL_\beta)\cdot\sI^k)}\right)\\
&=&-kr\tau_\beta (D)\cdot h^0(X, \sO_X(kL_\beta))+v(k),
\end{eqnarray*}
where
\[
v(k):=\sum_{j=1}^{kr\tau_\beta(D)}h^0(X, \sO_X(kL_\beta-jD)).
\]
By the same argument, 
\begin{eqnarray*}
\tilde{w}(k)&=&-\dim\left(\frac{H^0(D\times\A^1, p_1^*\sO_X(kL_\beta)|_D)}{H^0
(D\times\A^1, p_1^*\sO_X(kL_\beta)|_D\cdot(t^{kr\tau_\beta(D)}))}\right)\\
&=&-kr\tau_\beta(D)\cdot h^0(D, \sO_X(kL_\beta)|_D).
\end{eqnarray*}
Thus 
\[
\tilde{b}_0=-r\tau_\beta(D)\frac{(L_\beta^{\cdot n-1}\cdot D)}{(n-1)!}.
\]
We set $v(k)=v_0k^{n+1}+v_1k^n+O(k^{n-1})$. We calculate the values $v_0$ and $v_1$. 
Let $L_{\beta, i}$ and $D_i$ be the divisors on $X_i$ which are 
the push-forwards of $L_\beta$ and $D$, respectively. 
For $k\gg 0$ sufficiently divisible, by \cite[Remark 2.4 (i)]{KKL} and 
\cite[Proposition 4.1]{fjt}, $v(k)$ is equal to

\begin{eqnarray*}
&&\sum_{i=1}^m\sum_{j=kr\tau_{i-1}+1}^{kr\tau_i}h^0\left(X_i, \sO_{X_i}\left(
kL_{\beta, i}-jD_i\right)\right)\\
&=&\sum_{i=1}^m\Bigl(\frac{(kr)^{n+1}}{n!}\int_{\tau_{i-1}}^{\tau_i}
\left(\left((1/r)L_{\beta, i}-xD_i\right)^{\cdot n}\right)dx\\
&&-\frac{(kr)^n}{2\cdot (n-1)!}\int_{\tau_{i-1}}^{\tau_i}
\left(\left((1/r)L_{\beta, i}-xD_i\right)^{\cdot n-1}\cdot(K_{X_i}+D_i)\right)dx\Bigr)\\
&&+O(k^{n-1}).
\end{eqnarray*}

This implies that 
\begin{eqnarray*}
v_0&=&\frac{r^{n+1}}{n!}\sum_{i=1}^m\int_{\tau_{i-1}}^{\tau_i}
\left(\left((1/r)L_{\beta, i}-xD_i\right)^{\cdot n}\right)dx,\\
v_1&=&\frac{-r^n}{2\cdot(n-1)!}\sum_{i=1}^m\int_{\tau_{i-1}}^{\tau_i}
\left(\left((1/r)L_{\beta, i}-xD_i\right)^{\cdot n-1}\cdot(K_{X_i}+D_i)\right)dx.
\end{eqnarray*}

Thus we have
\begin{eqnarray*}
&&\DF_\beta((\sX, \sD), \sL_\beta)\\
&=&2(v_0a_1-v_1a_0)+(1-\beta)(a_0\tilde{b}_0-(v_0-r\tau_\beta(D)a_0)\tilde{a}_0)\\
&=&\frac{n\cdot r^n(L_\beta^{\cdot n})}{(n!)^2}
\sum_{i=1}^m\int_{\tau_{i-1}}^{\tau_i}(\beta-x)\left((-K_{X_i}-(1-\beta+x)
D_i)^{\cdot n-1}\cdot D_i\right)dx.
\end{eqnarray*}

\begin{lemma}[{cf.\ \cite[Theorem 5.2]{fjt}}]\label{vol_lem}
We have 
\[
\eta_\beta(D)=n\sum_{i=1}^m\int_{\tau_{i-1}}^{\tau_i}(\beta-x)\left((-K_{X_i}-(1-\beta+x)
D_i)^{\cdot n-1}\cdot D_i\right)dx.
\]
\end{lemma}

\begin{proof}
By \cite[Remark 2.4 (i)]{KKL}, we have
\[
\vol_X(-K_X-(1-\beta+x)D)=((-K_{X_i}-(1-\beta+x)D_i)^{\cdot n})
\]
for any $x\in[\tau_{i-1}, \tau_i]$. From partial integration, we have 
\begin{eqnarray*}
&&n\sum_{i=1}^m\int_{\tau_{i-1}}^{\tau_i}(\beta-x)\left((-K_{X_i}-(1-\beta+x)
D_i)^{\cdot n-1}\cdot D_i\right)dx\\
&=&\sum_{i=1}^m\Bigl(\left[(x-\beta)\vol_X
(-K_X-(1-\beta+x)D)\right]_{\tau_{i-1}}^{\tau_i}\\
&&-\int_{\tau_{i-1}}^{\tau_i}\vol_X(-K_X-(1-\beta+x)D)dx\Bigr)=\eta_\beta(D).
\end{eqnarray*}
We remark that $\vol_X(-K_X-(1-\beta+x)D)=0$ if $x\geq\tau_\beta(D)$. 
\end{proof}

Therefore we have obtained the following. 

\begin{thm}\label{thmthm}
Under Assumption \ref{fg_assump}, assume that $((X, D), -K_X-(1-\beta)D)$ is 
log K-stable $($resp.\ log K-semistable$)$ with cone angle $2\pi\beta$. 
Then $\eta_\beta(D)>0$ $($resp.\ $\geq 0)$ holds. 
\end{thm}

\begin{remark}\label{fjt_rmk}
If $\beta=1$, then the value $\eta_1(D)$ is noting but the value $\eta(D)$ in 
\cite[Definition 1.1]{fjt}. 
\end{remark}

\begin{corollary}\label{corcor}
Let $X$ be a normal projective variety which is log terminal, $D$ is a nonzero 
reduced Weil divisor on $X$ which is $\Q$-Cartier. 
Assume that $(X, (1-\beta)D)$ is klt and 
$-K_X-(1-\beta)D$ is ample for any $0<\beta\ll 1$. Moreover, we assume that 
$-K_X-D$ is big. 
Then for any $0<\beta\ll 1$ rational number, $((X, D), -K_X-(1-\beta)D)$ is not 
log K-semistable with cone angle $2\pi\beta$. 
\end{corollary}

\begin{proof}
We have $\eta_\beta(D)=\eta_+(\beta)-\eta_-$, where 
\begin{eqnarray*}
\eta_+(\beta)&:=&\beta\cdot\vol_X(-K_X-(1-\beta)D)-\int_{1-\beta}^1\vol_X(-K_X-xD)dx,\\
\eta_-&:=&\int_1^\infty\vol_X(-K_X-xD)dx.
\end{eqnarray*}
If $-K_X-D$ is big, then $\eta_-> 0$. On the other hand, 
$\lim_{\beta\to 0}\eta_+(\beta)$ is equal to zero. 
Thus the assertion follows from Theorem \ref{thmthm}. 
\end{proof}

Corollary \ref{intro_cor} is immediately obtained from Theorem \ref{agdg_thm} and 
Corollary \ref{corcor}.

\section{Examples}\label{ex_section}

We see some examples. 

\begin{example}\label{ex1}
Let $X$ be an $n$-dimensional Fano manifold, and let $D$ be a smooth divisor on $X$ 
with $-K_X\sim_\Q lD$ for some $l\in[1, n+1]\cap\Q$. Then $-K_X-(1-\beta)D$ is 
ample for any $\beta\in(0, 1]$. In this case, we have 
\[
\eta_\beta(D)=\frac{n}{n+1}\vol_X(-K_X-(1-\beta)D)\left(\beta-\frac{l-1}{n}\right).
\]
If $\beta<(l-1)/n$, then $\eta_\beta(D)<0$ holds.
\end{example}

\begin{remark}\label{rubinstein_rmk}
In Example \ref{ex1}, if $D\sim -K_X$ (i.e., $l=1$), then $\eta_\beta(D)>0$ for any 
$\beta\in(0, 1]$. Thus our argument does not give any destabilizing information in this 
case. Compare with \cite[\S 6]{CR}.
\end{remark}

\begin{example}\label{ex2}
Let $Y:=\pr_{\pr^1}(\sO_{\pr^1}\oplus\sO_{\pr^1}(1))$, $C$ be a 
section of the $\pr^1$-bundle $Y/\pr^1$ with $(C^{\cdot 2})=1$, 
$\pi\colon X\to Y$ be the blowing up along $p\in Y$ with $p\in C$, $E$ be the 
exceptional divisor of $\pi$, and set $D:=\pi^{-1}_*C$. 
Then $-K_X-(1-\beta)D$ is ample for any $\beta\in(0, 1]$. 
(The pair $(X, D)$ is nothing but \cite[(I8B.1)]{CR0}.) In this case, 
$\tau_1=\beta$, $X_2=Y$, $\tau_2=\tau_\beta(D)=1+\beta$ and 
\[
\vol_X(-K_X-xD)=\begin{cases}
-4x+7 & \text{if }x\in[0,1],\\
x^2-6x+8 & \text{if }x\in[1,2].
\end{cases}
\]
Thus $\eta_\beta(D)=2(\beta^2-2/3)$. 
For example, if we set $\beta:=1/2$, then $\eta_{1/2}(D)<0$ holds. 
In this case ($\beta=1/2$), we can check that $r:=2$ satisfies the condition in 
Remark \ref{BCHM_rmk} and the corresponding flag ideal $\sI$ is of the form 
\[
\sI=\sO_X(-3D-2E)+\sO_X(-2D-E)t^1+\sO_X(-D)t^2+(t^3).
\]
\end{example}

\end{document}